\documentclass{amsart}
\usepackage{graphicx}
\usepackage[british]{babel}
\usepackage{amsfonts}
\usepackage{amssymb}
\usepackage{amsmath}
\usepackage[all]{xy}
\newtheorem{thm}{Theorem}[section]
\newtheorem{theorem}{Theorem}[section]

\newtheorem{proposition}[thm]{Proposition}
\theoremstyle{definition}
\newtheorem{definition}[thm]{Definition}
\theoremstyle{remark}

\numberwithin{equation}{section}

\newcommand{\A}{\mathcal{A}}
\newcommand{\N}{\ensuremath{\mathbb{N}}}
\newcommand{\norm}{\ensuremath{\mathsf{norm}}}
\newcommand{\R}{\ensuremath{\mathbb{R}}}
\newcommand{\C}{\ensuremath{\mathcal{C}}}
\begin{document}

\title[]{From A-spaces to arbitrary spaces via spatial fibrous preorders}%
\author{N. Martins-Ferreira}%
\address{Escola Superior de Tecnologia e Gest\~ao\\ Centro Para o Desenvolvimento R\'apido e Sustentado do Produto\\ Instituto Poli\-t\'ecnico de Leiria, Leiria,
Portugal}
\email{martins.ferreira@ipleiria.pt}%

\thanks{Research supported by IPLeiria(ESTG/CDRSP) and FCT grant SFRH/BPD/43216/2008. Also by the FCT projects: PTDC/EME-CRO/120585/2010, PTDC/MAT/120222/2010 and PEst-OE/EME/UI4044/2013.}%

\begin{abstract}
The well known equivalence between preorders and Alexandrov spaces
is extended to an equivalence between arbitrary topological spaces
and spatial fibrous preorders, a new notion to be introduced.
\end{abstract}
\maketitle
\section{Introduction}
\label{intro} In modern terms, the main result in \cite{Alex}
establishes a categorical equivalence between preorders and
A-spaces. A preorder is simply a reflexive and transitive relation
while an A-space is a topological space in which any intersection of
open sets is open. The later trivially holds for finite topological
spaces and the equivalence between finite topological spaces and
finite preorders was used in \cite{JS1,JS2} to answer to open
problems in topological descent theory. In \cite{Erne}, p.61, Ern\'e
writes \emph{"Hence the question arises: How can we enlarge the
category of A-spaces on the one hand and the category of
quasiordered sets on the other hand, so that we still keep an
equivalence between the topological and the order-theoretical
structures, but many interesting 'classical' topologies are included
in the extended definition?"} and proposes the two notions of
B-space and C-space \cite{Erne}.

With a different motivation, and not being restricted to the
\emph{order-theoretical structures}, we propose a new structure,
which we call fibrous preorder, and generalizes the one of a
preorder. With the appropriate morphisms, called fibrous morphisms,
and a suitable equivalence between them, we observe that the
category Top, of topological spaces and continuous maps, is sitting
in the middle
\[\xymatrix{Preord\ar[r]&Top\ar[r]&FibPreord}\]
of the category of preorders and the category of fibrous preorders.

The main result of this work is the description of the subcategory
of fibrous preorders which is equivalent to the category of
topological spaces. Inspired by what is called spatial frames in
point-free-topology (see e.g. \cite{PP}), the fibrous preorders
arising in this way are called spatial fibrous preorders.

A fibrous preorder is a generalization of a preorder. It was
obtained while looking for a simple description of topological
spaces in terms of internal categorical structures. By an internal
categorical structure we mean a structure which can be defined in an
arbitrary category with finite limits --- as for instance the notion
of internal category or internal groupoid, internal preorder or
internal equivalence relation. A detailed description on these
topics can be found for instance in \cite{BB}.

This work is organized as follows: in the section \ref{secfibpre} we
describe the category of fibrous preorders, by defining its objects
and morphisms and an equivalence relation on each hom-set of fibrous
morphisms; that induces an equivalence between fibrous preorders
that we make explicit; at the end we recover the classical
Alexandrov theorem stating that every A-space is equivalent to a
preorder. In section \ref{main} we introduce the notion of spacial
fibrous preorder and prove that (up to equivalence) it defines a
subcategory of the category of fibrous preorders and moreover that
it is isomorphic to the category of topological spaces. In section
\ref{examples} we provide some examples to illustrate how the use of
spacial fibrous preorders can be used to work with topological
spaces described by systems of open neighbourhoods.

\section{Fibrous preorders and fibrous morphisms}\label{secfibpre}

The following definition is a generalization of the notion of
preorder, i.e. a reflexive and transitive relation. The word fibrous
is a derivation of the word fibre and it is motivated by the
presence of a morphism $p\colon{A\to B}$ (see below), suggesting
that $A$ may be considered as a fibre over the base $B$, moreover
when $p$ is an isomorphism the classical notion of preorder is
recovered.

\begin{definition}\label{defA} A fibrous preorder is a sequence \[\xymatrix{R\ar[r]^{\partial}&A\ar[r]^{p}&B}\] in
which $A$ and $B$ are sets, $p$ and  $\partial$ are maps,
$R\subseteq A\times B$ is a binary relation (and as usual we simply
write $(a,b)\in R$ as $aRb$) such that the following conditions
hold:
\begin{enumerate}
\item[(F1)] $p\partial(a,b)=b$;
\item[(F2)] $aRp(a)$;
\item[(F3)] $\partial(a,b)Ry\Rightarrow aRy$;
\end{enumerate}
for every $a\in A$ and $b,y\in B$ with $aRb$.\end{definition}

\begin{definition}\label{defB}
Let $\mathbf{A}=(\xymatrix{R\ar[r]^{\partial}&A\ar[r]^{p}&B})$ and
$\mathbf{A'}=(\xymatrix{R'\ar[r]^{\partial'}&A'\ar[r]^{p'}&B'})$ be
two fibrous preorders. A fibrous morphism between $\mathbf{A}$ and
$\mathbf{A'}$ is a pair $(f,f^{*})$ with $f\colon{B\to B'}$ a map
from $B$ to $B'$ and $f^{*}\colon{A'_{f}\to A}$ a map from
\[A'_f=\{(a',b)\in A'\times B\mid p'(a')=f(b)\}\] to A such that
\begin{equation}\label{F7}pf^{*}(a',b)=b\end{equation} and
\begin{equation}\label{F8}f^{*}(a',b)Ry\Rightarrow a'R'f(y)\end{equation} for all $a'\in A'$ and $b,y\in
B$ with $p'(a')=f(b)$.
\end{definition}

In other words a fibrous morphism from $\mathbf{A}$ to $\mathbf{A'}$
is a map $f$ from $B$ to $B'$ together with a span
\[\xymatrix{A'&A'_f\ar[l]_{\pi_1}\ar[r]^{f^{*}}&A}\] such that the
following diagram (in which the top line is to be considered as a
single composite arrow) is a pullback diagram
\[\xymatrix{A'_f\ar[r]^{f^{*}}\ar[d]_{\pi_1}&A\ar[r]^{p}&B\ar[d]^{f}\\A'\ar[rr]^{p'}&&B'}\]
and moreover the condition $(\ref{F8})$ is satisfied. Note that the
commutativity of the previous diagram is equivalent to condition
$(\ref{F7})$.

Now, if $(g,g^*)$ is another fibrous morphism, say from
$\mathbf{A'}$ to $\mathbf{A''}$, then the composition
$(g,g^{*})\circ(f,f^{*})$ is computed as
\[(g,g^{*})\circ(f,f^{*})=(gf,f^{*}g_{f}^{*})\] with
$f^{*}g_{f}^{*}(a'',b)=f^{*}(g^{*}(a'',f(b)),b)$. The following
diagram illustrates how the above formula is obtained (simply
complete the diagram by inserting the upper left pullback square)
and it also shows that this formula is associative up to a canonical
isomorphism of pullbacks.
\[\xymatrix{A''_{gf}\ar[r]^{g_{f}^{*}}\ar[d]_{\pi_1}\ar@{}[dr]|{p.b.}&A'_f\ar[r]^{f^{*}}\ar[d]_{\pi_1}&A\ar[r]^{p}&B\ar[d]^{f}\\A''_g\ar[r]^{g^{*}}\ar[d]_{\pi_1}&A'\ar[rr]^{p'}&&B'\ar[d]^{g}\\A''\ar[rrr]^{p''}&&&B''}\]

It is a simple calculation to check that it is well defined, that
is, condition $(\ref{F8})$ is satisfied.

We will consider the category FibPreord of fibrous preorders and
fibrous morphisms with the following identification of parallel
fibrous morphisms.

\begin{definition}\label{defC} Two parallel fibrous morphisms $(f,f^{*})$ and  $(g,g^{*})$  are said to be equivalent if and only if
$f=g$.
\end{definition}

This equivalence of morphisms immediately gives the following
equivalence between two objects: we identify two fibrous preorders
whenever they have the same base object and the identity map is
fibrous in both directions.

\begin{proposition}Two fibrous preorders $(R,A,B,p,\partial)$ and
$(R',A',B',p',\partial')$ are equivalent
\[(R,A,B,p,\partial)\sim(R',A',B',p',\partial)\] if and only if $B=B'$ and there
exist two maps
\[\xymatrix{A\ar@<.5ex>[r]^{\varphi}\ar[d]_{p}&A'\ar@<.5ex>[l]^{\gamma}\ar[d]^{p'}\\B\ar@{=}[r]&B'}\]
such that $p'\varphi=p$, $p\gamma=p'$ and
\begin{equation}\label{F9}\varphi(a)R'b\Rightarrow aRb\end{equation}\begin{equation}\label{F10}\gamma(a')Rb\Rightarrow
a'R'b\end{equation}for every $a\in A$, $a'\in A'$ and $b\in B$.
\end{proposition}
\begin{proof} Straightforward.
\end{proof}

A preorder $(B,\leq)$ is in particular a fibrous preorder with
$A=B$, $p=1_B$, $xRy$ if and only if $x\leq y$ and
$\partial(x,y)=y$. In fact we have more.

\begin{proposition}\label{prop_umap}
There is an embedding of the category of preorders into the category
of fibrous preorders and moreover an object $(R,A,B,p,\partial)$ is
(up to equivalence) in the image of the embedding if and only if
there exists a map \[u\colon{B\to A}\] such that $pu=1_B$ and
\[up(a)Ry\Rightarrow aRy,\] for every $a\in A$ and $b\in B$.
\end{proposition}
\begin{proof} Clearly if $(B,\leq)$ is a preorder then $(R,A,B,p,\partial)$ with
$A=B$, $p=1_B=1_A$, $xRy$ if and only if $x\leq y$ and
$\partial(x,y)=y$ is a fibrous preorder. Also every  monotone map
gives a fibrous morphism.

Conversely, let $(R,A,B,p,\partial)$ be a fibrous preorder. If there
exists a map \[u\colon{B\to A}\] such that $pu=1_B$ and
\[up(a)Ry\Rightarrow aRy,\] for every $a\in A$ and $b\in B$, then it
is equivalent to \[(R^{\circ},B,B,1_B,\partial^{\circ})\] with
\[xR^{\circ}y\Leftrightarrow u(x)Ry\] and $\partial^{\circ}(x,y)=y$.
First let us observe that  $R^{\circ}$ is a preorder, so that
$(R^{\circ},B,B,1_B,\partial^{\circ})$ is well defined. Indeed
$xR^{\circ}x\Leftrightarrow u(x)Rx\Leftrightarrow u(x)Rpu(x)$ which
holds by condition $(F2)$. This proves reflexivity. For
transitivity, suppose we have $xR^{\circ}y$ and $yR^{\circ}z$ that
is $u(x)Ry$ and $u(y)Rz$, hence we also have
$up(\partial(u(x),y))Rz$ and then it follows $\partial(u(x),y)Rz$
from which we conclude $u(x)Rz$, proving that $xR^{\circ}z$ as
desired. It is also clear that we have maps
\[\xymatrix{A\ar@<.5ex>[r]^{p}\ar[d]_{p}&B\ar@<.5ex>[l]^{u}\ar[d]^{1_B}\\B\ar@{=}[r]&B}\]
with \[p(a)R^{\circ}b\Rightarrow up(a)Rb\Rightarrow aRb\] and
$u(x)Rb\Rightarrow xR^{\circ}b$.
\end{proof}

There is also a functor from the category of topological spaces to
the one of fibrous preorders which will be used in the next section.

\begin{proposition}\label{exampletopspaces}
If $(B,\tau)$ is a topological space then the structure
\begin{align*}
 A=\{(U,x)\mid x\in U \in \tau\}\\
 p(U,x)=x\\
 (U,x)Ry\Leftrightarrow y\in U\\
 \partial((U,x),y)=(U,y)
\end{align*}
defines a fibrous preorder. Moreover, if
$f\colon{(B,\tau)\to(B',\tau')}$ is a continuous map then
\[(f,f^{*})\colon{(R,A,B,p,\partial )\to(R',A',B',p',\partial' )},\] with
$f\colon{B\to B'}$ the underlying map and
\[f^{*}((U',x'),y)=(f^{-1}(U'),y)\] for all $(U',x')\in A'$ and
$y\in B$ with $x'=f(y)$, is a fibrous morphism between fibrous
preorders. \end{proposition}
\begin{proof}Straightforward verification.\end{proof}

In the next section we describe the fibrous preorders arising from a
topological space, called spatial fibrous preorders and prove that
the category of topological spaces is isomorphic to the category of
spatial fibrous preorders. Before that we illustrate how the
classical result of \cite{Alex} can be obtained via this new
setting.

\begin{proposition}[Alexandrov, \cite{Alex}]Let $(B,\tau)$ be a topological space and consider the fibrous
preorder, say $F(B,\tau)$, described in Proposition
\ref{exampletopspaces}. It is an Alexandrov space if and only if
$F(B,\tau)$ is equivalent to a preorder.\end{proposition}
\begin{proof}We only observe that if $(B,\tau)$ is an Alexandrov space then
there exists a map
\[u\colon{B\to A}\] assigning to each point $x\in B$ the element
$(\theta_x,x)\in A$ with $\theta_x$ the intersection of all open
neighbourhoods of $x$, moreover this map satisfies the requirements
of proposition \ref{prop_umap}.
\end{proof}

\section{The main result}\label{main}

The so called spatial frames are the frames that are isomorphic to
the topology of some space (see e.g. \cite{PP}). Here our main
result is the description of the full subcategory of fibrous
preorders and fibrous morphisms (with equivalent morphisms
identified) which is equivalent to the category of topological
spaces.

\begin{definition} A fibrous preorder $(\xymatrix{R\ar[r]^{\partial}&A\ar[r]^{p}&B})$ is said to be \emph{spatial} when there exists $s\colon{B\to A}$ and $m\colon{A\times_B A\to A}$ with
 $A\times_B A=\{(a,a')\in A\times A\mid p(a)=p(a')\}$ such that
\begin{enumerate}
\item[(F4)] $ps(y)=y$;
\item[(F5)] $pm(a,a')=p(a)=p(a')$;
\item[(F6)] $m(a,a')Ry \Rightarrow aRy \& a'Ry$;
\end{enumerate}
for every $a,a'\in A$ and $y\in B$ with
$p(a)=p(a')$.\end{definition}

Observe that because we are identifying fibrous morphisms as in
Definition \ref{defC}, the notion of spatial fibrous preorder is a
property of a given fibrous  preorder and not an extra structure.
That is to say that the category of spatial fibrous preorders and
(equivalent) fibrous morphisms is a full subcategory of FibPreord.

\begin{theorem}
There are functors $F$ and $G$ between the category of spatial
fibrous preorders and topological spaces
\[\xymatrix{SpFibPreord\ar@<.5ex>[r]^(.65){F}&Top\ar@<.5ex>[l]^(.35){G}}\] such that $FG=1$ and $GF\sim
1$.
\end{theorem}
\begin{proof}
The functor $G$ is defined as in Proposition \ref{exampletopspaces},
together with $s(x)=(B,x)$ and $m((U,x),(V,x))=(U\cap V,x)$. The
functor $F$ associates to each spatial fibrous preorder
\[\xymatrix{R\ar[r]^{\partial}&A\ar@<.5ex>[r]^{p}&B\ar@<.5ex>[l]^{s}},\quad
\xymatrix{A\times_{\langle p,p\rangle}A\ar[r]^(.65){m}&A}\] the
topological space $(B,\tau)$ in which $\tau$ is defined by
\begin{equation}\label{eqopen}\mathcal{O}\in\tau\Leftrightarrow \forall y\in \mathcal{O}\
\exists a\in A,\, p(a)=y,\, N(a)\subseteq\mathcal{O},\end{equation}
with $N(a)=\{y\in B\mid aRy\}$. If $(f,f^*)$ is a morphism in
$SpFibPreord$ then $F(f,f^*)=f$. In order to see that the functor
$F$ is well defined we observe:
\begin{enumerate}
\item[(a)] The empty set is in $\tau$. Indeed it is an immediate consequence of
$(\ref{eqopen})$.
\item[(b)] If $\mathcal{O},\mathcal{O'}\in \tau$ then
$\mathcal{O}\cap\mathcal{O'}\in\tau$. Indeed if $x\in
\mathcal{O}\cap\mathcal{O'}$ then by $(\ref{eqopen})$ there exist
$a,a'\in A$ such that $p(a)=p(a')=x$ and $N(a)\subseteq
\mathcal{O}$, $N(a')\subseteq \mathcal{O'}$. Using $(F5)$ and $(F6)$
we obtain $m(a,a')$ such that $pm(a,a')=x$ and $N(m(a,a'))\subseteq
N(a)\cap N(a')$, yielding  the desired result that for every
$x\in\mathcal{O}\cap\mathcal{O'}$ there exists $m(a,a')\in A$ such
that $pm(a,a')=x$  and
$N(m(a,a'))\subseteq\mathcal{O}\cap\mathcal{O'}$. This proves that
$\mathcal{O}\cap\mathcal{O'}$ is in $\tau$.
\item[(c)] The fact that $p$ is surjective $(F4)$ implies (in fact is equivalent to the fact) that
$B$ is in $\tau$.
\item[(d)] Again by definition of $\tau$ it is easy to see that it
is closed under arbitrary unions.
\end{enumerate}
Concerning morphisms we have to show that if $(f,f^*)$ is a morphism
in $SpFibPreord$ then $f$ is a continuous map from $(B,\tau)$ to
$(B',\tau')$, assuming that $\tau$ and $\tau'$ are obtained as in
$(\ref{eqopen})$. Suppose $\mathcal{O'}\in\tau'$, we shall prove
$f^{-1}(\mathcal{O'})\in\tau$. Given any $y\in
f^{-1}(\mathcal{O'})$, because $\mathcal{O'}\in\tau'$ and
$f(y)\in\mathcal{O'}$, by $(\ref{eqopen})$ there exists $a'\in A'$
such that $p(a')=f(y)$ and $N'(a')\subseteq\mathcal{O'}$. Now, the
very structure of fibrous morphism gives us an element
$f^{*}(a',y)\in A$ such that (see $(\ref{F7})$ and $(\ref{F8})$)
\[pf^{*}(a',y)=y\] and \[N(f^{*}(a',y))\subseteq f^{-1}(N'(a'))\subseteq
f^{-1}(O'),\] proving thus that  $f^{-1}(O')\in\tau$, whenever
$O'\in\tau'$. And hence $f$ is a continuous map.

It is also immediate to observe that $FG=1_{Top}$. Indeed
\[FG(B,\tau)=(B,\bar{\tau})\] with \[\mathcal{O}\in\bar{\tau}\Leftrightarrow \forall y\in \mathcal{O}, \exists u\in\tau, y\in u\subseteq
\mathcal{O}\] and it is clear that $\bar{\tau}=\tau$.

It remains to prove $GF\simeq 1$. To do that we will show the
existence of two maps
\[\xymatrix{A\ar@<.5ex>[rr]^{\varphi}\ar[d]_{p}&&
GF(A)=\bar{A}\ar@<.5ex>[ll]^{\gamma}\ar[d]^{\bar{p}}\\B\ar@{=}[rr]&&B}\]
such that the diagram above commutes and for every $a\in A$,
$(u,x)\in\bar{A}$ and $b\in B$,
\[\varphi(a)\bar{R}b\Rightarrow aRb\]\[\gamma(u,x)Rb\Rightarrow
(u,x)\bar{R}b.\] The map $\varphi$ is given by
\[\varphi(a)=(N(a),p(a))\] while (assuming the axiom of choice)
\[\gamma(u,x)=a_{u,x}=a_x\] where $a_x\in A$ is any element of $A$ such that
$p(a_x)=x$ and $N(a_{u,x})\subseteq u$, which exists by definition
of $(u,x)\in \bar{A}$. Recall that $(u,x)\in \bar{A}$ if and only if
$x\in u$ and  $u\subseteq B$ is such that \[\forall y\in u,\exists
a\in A,\quad p(a)=y, N(a)\subseteq u.\]

In order to prove that $\varphi(a)$ is well defined we observe: by
$(F2)$, $aRp(a)$, and so $p(a)\in N(a)$; now suppose $y\in N(a)$,
this means $aRy$, and if we put $a'=\partial(a,y)$ then by $(F1)$
$p(a')=y$ and by $(F3)$ we know that \[N(a')\subseteq N(a),\]
showing that $\varphi$ is well defined.

Finally we observe that \[(N(a),p(a))\bar{R}b\Leftrightarrow b\in
N(a) \Leftrightarrow aRb\] and \[a_{u,x}Rb\Leftrightarrow b\in
N(a_{u,x})\subseteq u\Rightarrow (u,x)\bar{R}b,\] which concludes
the proof.
\end{proof}

\section{Some examples}\label{examples}

We conclude with a list of examples to illustrate how this notion of
spatial fibrous preorder can be used to work with arbitrary
topological spaces described by basic neighbourhood relations.

\subsection*{Normed vector spaces}
Let $B=(B,+,0)$ be an abelian group and $I\subseteq B$ a subset of
$B$ together with a map $h\colon{I\to \N}$ such that:
\begin{enumerate}
\item $0\in I$;
\item it $nn'a\in I$ then $na,n'a\in I$;
\item if $na\in I$ and $h(a)a'\in I$ then $n(a+a')\in I$.
\end{enumerate}
We construct a spatial fibrous preorder as follows:
\begin{eqnarray*} A=\N\times B \\ p(n,x)=x,\quad s(x)=(1,x),\quad m(n,n',x)=(nn',x)\\
 (n,x)Ry\Leftrightarrow n(x-y)\in I\\
 \partial((n,x),y)=\left\{\begin{array}{c}
(h(x-y),y)\text{ if } x\neq y \\
(n,x)\text{ if } x=y
\end{array}\right.
\end{eqnarray*}
A concrete example is the case when $B$ is a normed vector space
with $I=\{x\in B\mid \norm(x)<1\}$ and $h(x)\in \N$ is such that
$\frac{1}{h(x)}<\frac{1}{k}-\norm(x)$ with $k$ the unique natural
number such that \[\frac{1}{k+1}\leq \norm(x)<\frac{1}{k}\] if
$\norm(x)\neq 0$.
\subsection*{Metric Spaces}
The intuitive idea of working with metric spaces with open balls of
radius $\frac{1}{n}$ for some natural number $n$ may be formalized
in terms of spatial fibrous preorders in the following way.  Let
$(B,d)$ be a metric space and consider
\[\begin{aligned}
A=\mathbb{N}\times B\\
p(n,x)=x,\quad s(x)=(1,x)\\
m((n,x),(n',x))=(nn',x)\\
(n,x)Ry\Leftrightarrow d(x,y)<\frac{1}{n}\\
\partial((n,x),y)=(k,y)
\end{aligned}\]
with $k\in\mathbb{N}$ any number greater than $\frac{n}{1-nd(x,y)}$.

This shows the existence of a functor (actually and embedding) from
the category of metric spaces and continuous maps into the category
of spatial fibrous preorders and fibrous morphisms. If
$f\colon{(B,d)\to (B',d')}$ is a continuous map we may define
$f^*((n,f(y),y))$ as the pair $(k,y)$ with $k$ any natural number
such that for every $z\in Z$
\[d(z,y)<\frac{1}{k}\Rightarrow d'(f(z),f(y))<\frac{1}{n}.\]

We may now ask for a characterization of those spatial fibrous
preorders which arise from a metric space in the same way as above.
As it is well-known in point-set topology there is not a simple
answer to that question. Nevertheless the notion of a \emph{natural}
space is a good substitute to the one of a metric space.

\subsection*{\emph{Natural} spaces}

Generalizing the construction for metric spaces from above we
observe that more in general, every map $N\colon{\N\times B\to
P(B)}$ such that
\begin{enumerate}
\item $\forall n\in \N, \forall x\in B,\quad x\in N(n,x)$
\item $N(nn',x)\subseteq N(n,x)\cap N(n',x)$
\item $N(n,x)\subseteq \{y\in B\mid \exists n'\in\N,\quad N(n',y)\subseteq N(n,x)\}$
\end{enumerate}
gives a spatial fibrous  preorder as follows: $A=\N\times B$,
$p(n,x)=x$, $s(x)=(1,x)$, $m(n,n',x)=(nn',x)$,
\[(n,x)Ry\Leftrightarrow y\in N(n,x), \] and $\partial(n,x,y)=n'$
where $n'$ is such that $N(n',y)\subseteq N(n,x)$ which exists by
definition of $N$.

For the purpose of this note a  topological space is said to be
\emph{natural} if it admits a base of neighbourhoods of the form
above. It is clear from the above that every metrizable space is
natural.

In this case, a morphism $f\colon{B\to B'}$ is continuous if for
every $n\in\N$ and $y\in B$ there exists $k\in\N$ such that
$f(N(k,y))\subseteq N'(n,f(y))$.

The particular the case of metric spaces is recaptured by letting
$N(n,x)=\{y\in B\mid d(x,y)<\frac{1}{n}\}$.

However, as it is well known, not every natural space is metrizable.

\subsection*{Tangent disk topology}
An example of a well known non-metrizable space which is natural in
the sense above is the so-called tangent disk topology. In this case
$B=\{(x,y)\in\R^2\mid y\geq 0\}$ and
\[N(n,(x,y))=\left\{
\begin{array}{lcr}
    \{(x_1,x_2)\mid d((x_1,x_2),(x,y))<\frac{1}{n}\} & \text{if} & y>0\\
    \{(x_1,x_2)\mid d((x_1,x_2),(x,y))<\frac{1}{n}\}\cup \{(x,y)\} & \text{if} &
    y=0.
\end{array}
\right.\]

\subsection*{The Cantor set}
The Cantor Set is another well known concrete example that fits in
the setting of natural spaces: in this case we let \[B=\{u\mid
u\colon{\N\to \{0,2\}}\}\] and
\[N(n,u)=\{w\in B\mid w(i)=u(i),\quad i\leq n\}.\]

More generally we may consider as $B$ any set of the form $\{u\mid
u\colon{\N\to X}\}$ with $X$ an arbitrary set.

\subsection*{The $p$-adic topology} The $p$-adic topology on the set of integers is obtained as \[N(n,x)=\{z\in\mathbb{Z}\mid z=x+kp^n,\quad
k\in\mathbb{Z}\}\] with $B=\mathbb{Z}$.

 If instead of a  map $N\colon{\N\times B\to P(B)}$ we
consider a family of binary relations $R_n$ over $B$, then we have
examples of the following type, with $I=\N$.

\subsection*{Indexed families of preorders}

A more general example is obtained as follows. Let $I$ be a monoid,
$B$ a set, $(R_i)_{i\in I}$ a family of binary relations
$R_i\subseteq B\times B$, and $(\partial_i\colon{R_i\to I})$ a
family of maps such that:
\begin{enumerate}
\item $xR_i x$
\item $xR_{ij}y \Rightarrow xR_i y \& xR_j y$
\item $xR_i b \& bR_{\partial_i(x,b)}y \Rightarrow xR_i y$
\end{enumerate}
for all $i,j\in I$ and $x,y,b\in B$.

In this case we construct a fibrous preorder as follows: $A=I\times
B$, $p(i,x)=x$, $s(x)=(1,x)$, $m(i,j,x)=(ij,x)$,
\[(i,x)Ry\Leftrightarrow xR_i y\] and
$\partial(i,x,y)=(\partial_i(x,y),y)$ if $xR_i y$.

For morphisms (from $(\partial_i\colon{R_i\to I})_i\in I$ to
$(\partial'_{i'}\colon{R'_{i'}\to I'})_{i'}\in I'$) we have a map
$f\colon{B\to B'}$ and a family of maps $f_j\colon{B\to I'}_j\in I'$
such that \[xR_{f_j(x)}y\Rightarrow f(x)R'_jf(y).\]

\section{Conclusion}

In this note we introduced the notions of (spatial) fibrous preorder
and fibrous morphism, showing that the category of topological
spaces is the quotient category of the category of spatial fibrous
preorders, obtained by identifying two fibrous morphisms whenever
they have the same underlying map. The examples  show that this
notion provides a convenient setting to work with the intuitive
notion of base of open neighbourhoods. However, as explained in the
introduction, the main motivation that leads to the definition of
fibrous preorder was the purpose of finding a purely categorical
definition of topological space. Future work (\cite{NMF_IFP}) will
specify the internal version of a fibrous preorder, by replacing the
relation $R\subseteq A\times B$ with a jointly monomorphic pair of
morphisms and by giving the appropriate translation of axioms
(F1)-(F3) and (F4)-(F6). In particular, the additional structure of
spatial fibrous preorder is nothing but a comonoid structure in the
monoidal category of fibrous preorders and fibrous morphisms, with
an appropriate tensor product. Further studies will then take place
in $FibPreOrd(\C)$ and $SpFibPreord(\C)$ for an arbitrary category
$\C$ with finite limits. For instance if $\C$ is the category of
finite sets then $Preord(\C)\simeq SpFibPreord(\C)\simeq
FibPreord(\C)$, as easily follows from Proposition \ref{prop_umap}.

\bibliographystyle{amsplain}

\end{document}